\documentclass[11pt]{article}
\usepackage{amssymb, amsmath,amsthm,mathtools}
\usepackage{xcolor}
\usepackage{graphicx}
\usepackage[utf8]{inputenc}
 \newtheorem{theorem}{Theorem}
\newtheorem{proposition}{Proposition}
 \newtheorem{corollary}{Corollary}
 \newtheorem{lemma}{Lemma}
 
\newtheorem{conjecture}{Conjecture}
 \theoremstyle{remark}
 
 \newtheorem{remark}{Remark}
\newtheorem{?}{Question}
\renewcommand{\epsilon}{\varepsilon}

\newcommand{\R}{\mathbb R}
\newcommand{\C}{\mathbb C}

\title{On the location of chromatic zeros of series-parallel graphs}
\author{Ferenc Bencs\thanks{Funded by the Netherlands Organisation of Scientific Research (NWO): VI.Vidi.193.068} 
\quad Jeroen Huijben\thanks{Funded by the Netherlands Organisation of Scientific Research (NWO) Mathematics Cluster PhD position 2017 WC.V17.012.}
\quad 
Guus Regts\thanks{Funded by the Netherlands Organisation of Scientific Research (NWO): VI.Vidi.193.068}\\
\small Korteweg de Vries Institute for Mathematics,  \\[-0.8ex]
\small University of Amsterdam,\\[-0.8ex] 
\small the Netherlands.\\
\small\tt \{ferenc.bencs,jeroenhuijben95,guusregts\}@gmail.com
}
\date{\today}
\begin{document}
\maketitle

\begin{abstract}
In this paper we consider the zeros of the chromatic polynomial of series-parallel graphs. 
Complementing a result of Sokal, giving density outside the disk \mbox{$|q-1|\leq1$},
we show density of these zeros in the half plane $\Re(q)>3/2$ and we show there exists an open region $U$ containing the interval $(0,32/27)$ such that $U\setminus\{1\}$ does not contain zeros of the chromatic polynomial of series-parallel graphs.

We also disprove a conjecture of Sokal by showing that for each large enough integer $\Delta$ there exists a series-parallel graph for which all vertices but one have degree at most $\Delta$ and whose chromatic polynomial has a zero with real part exceeding $\Delta$.
\end{abstract}

\section{Introduction}
Recall that the chromatic polynomial of a graph $G=(V,E)$ is defined as
\[
Z(G;q):=\sum_{F\subseteq E}(-1)^{|F|}q^{k(F)},
\]
where $k(F)$ denotes the number of components of the graph $(V,F)$.
We call a number $q\in \mathbb{C}$ a \emph{chromatic zero} if there exists a graph $G$ such that $Z(G;q)=0.$

About twenty years ago Sokal~\cite{Sokaldense} proved that the set of chromatic zeros of all graphs is dense in the entire complex plane. 
In fact, he only used a very small family of graphs to obtain density.
In particular, he showed that the chromatic zeros of all generalized theta graphs (parallel compositions of equal length paths) are dense outside the disk $B_1(1)$. (We denote for $c\in \mathbb{C}$ and $r>0$ by $B_r(c)$ the closed disk centered at $c$ of radius $r$.)
Extending this family of graphs by taking the disjoint union of each generalized theta graph with an edge and connecting the endpoints of this edge to all other vertices, he then obtained density in the entire complex plane.

As far as we know it is still open whether the chromatic zeros of all planar graphs or even series-parallel graphs are dense in the complex plane.
Motivated by this question and Sokal's result we investigate in the present paper what happens inside the disk $B_1(1)$ for the family of series-parallel graphs. See Section~\ref{sec:SP} for a formal definition of series-parallel graphs.
Our first result implies that the chromatic zeros of series-parallel are \emph{not} dense in the complex plane.

\begin{theorem}\label{thm:(0,32/27)}
There exists an open set $U$ containing the open interval $(0,32/27)$ such that $Z(G;q)\neq 0$ for any $q\in U\setminus \{1\}$ and for all series-parallel graphs $G$.
\end{theorem}
We note that the interval $(0,32/27)$ is tight, as shown in~\cite{jackson1993zero,thomassen1997zero}. In fact,  Jackson~\cite{jackson1993zero} even showed that there are no chromatic zeros in the interval $(1,32/27)$.
Unfortunately, we were not able to say anything about larger families of graphs and we leave open as a question whether Theorem~\ref{thm:(0,32/27)} is true for the family of all planar graphs for example.

In terms of chromatic zeros of series-parallel graphs inside the disk $B_1(1)$ we have found an explicit condition, Theorem~\ref{thm:to escape or not to escape} below, that allows us to locate many zeros inside this disk.
Concretely, we have the following results.

\begin{theorem}\label{thm:q>32/27}
Let $q>32/27$. Then there exists $q'\in \mathbb{C}$ arbitrarily close to $q$ and a series-parallel graph $G$ such that $Z(G;q')=0$.
\end{theorem}
This result may be seen as a a variation on Thomassen's result~\cite{thomassen1997zero} saying that real chromatic zeros (of not necessarily series-parallel graphs) are dense in $(32/27,\infty)$. 

Another result giving many zeros inside $B_1(1)$ is the following.
\begin{theorem}\label{thm:Re(q)>3/2}
The set of chromatic zeros of all series-parallel graphs is dense in the set $\{q\mid \Re(q)> 3/2\}.$
\end{theorem}
After inspecting our proof of Theorem~\ref{thm:Re(q)>3/2} (given in Section~\ref{sec:active}) it is clear that one can obtain several strengthenings of this result. 
Figure~\ref{fig:zeroes+zerofree} below shows a computer generated picture displaying where chromatic zeros of series-parallel graphs can be found as well as the zero-free region from Theorem~\ref{thm:(0,32/27)}.

\begin{figure}[ht] 
    \centering
    \includegraphics[width=12cm]{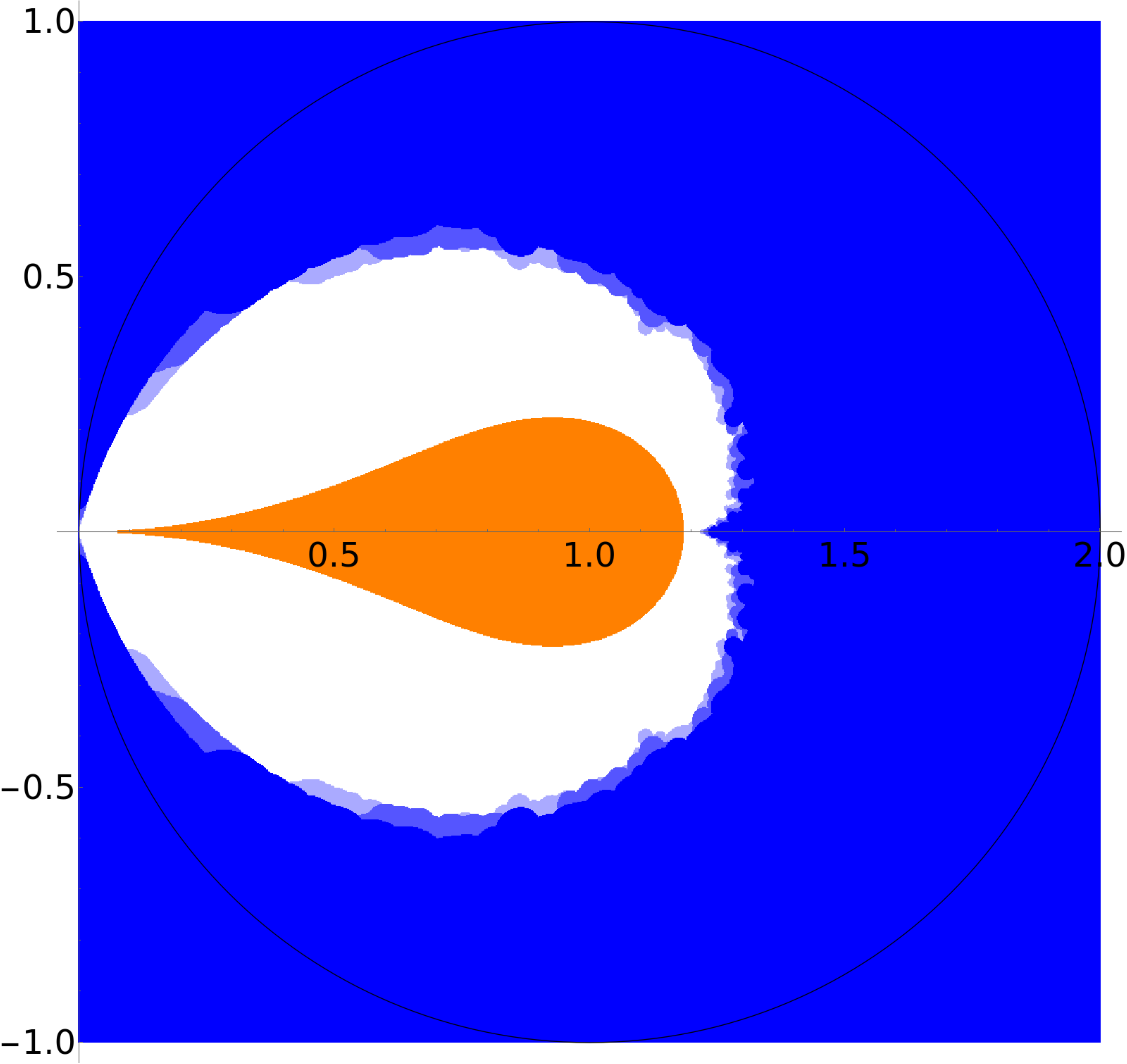}
    \caption{A pixel-picture of chromatic zeros and zero-free regions for series-parallel graphs, with a resolution of $1001\times 1001$ pixels. Every orange pixel represent a provably zero-free value of $q$, while every blue pixel represents a value of $q$ in the closure of the set of all chromatic zeros of series-parallel graphs. The region depicted in the picture ranges from $-i$ to $2+i$. We refer to Section~\ref{sec:questions} for more details concerning the shading.}
    \label{fig:zeroes+zerofree}
\end{figure}

We next restrict our attention to a subclass of series-parallel graphs.
A \emph{leaf joined tree} is a graph $\hat{T}$ obtained from a rooted tree $(T,v)$ by identifying all its leaves except possibly $v$ into a single vertex.
A while ago Sokal conjectured~\cite[Conjecture 9.5']{Sokalsurvey} that for each integer $\Delta\geq 3$ the chromatic zeros of all graphs all of whose vertices have degree at most $\Delta$ except possibly one vertex are contained in the half plane $\{q\mid \Re(q)\leq \Delta\}$.
For $\Delta=3$ this conjecture was disproved by Royle, as Sokal mentions in footnote 31 in~\cite{Sokalsurvey}.
Here we show that this is no coincidence, as we disprove this conjecture for all $\Delta$ large enough.
\begin{theorem}\label{thm:leaf joined}
There exists $\Delta_0>0$ such that for all integers $\Delta\geq \Delta_0$ there exists a leaf joined tree $\hat T$ obtained from a tree $T$ of maximum degree $\Delta$ such that $\hat{T}$ has a chromatic zero $q$ with $\Re(q)>\Delta$.
\end{theorem}
The proof of this theorem, together with some explicit calculations, also allows us to find such chromatic zeros for $4\leq\Delta\leq45$. Table~\ref{tab:over_Delta} in Section~\ref{sec:questions} records values of $q$, which are accumulation points of chromatic zeros of leaf joined trees, corresponding with the given $\Delta$.
% \[\begin{array}{c|c|c}
% \Delta & q & \text{type} \\\hline
% 4 & 4.027+0.783i & 3,2\\
% 5 & 5.088+0.836i & 4,3\\
% 6 & 6.132+0.881i & 5,4\\
% 7 & 7.058+1.521i & 6,4\\
% 8 & 8.120+1.577i & 7,5\\
% 9 & 9.012+2.194i & 8,5\\
% 10 & 10.084+2.256i & 9,6\\
% 11 & 11.147+2.314i & 10,7\\
% 12 & 12.038+2.928i & 11,7\\
% 13 & 13.109+2.990i & 12,8\\
% 14 & 14.173+3.049i & 13,9\\
% 15 & 15.063+3.662i & 14,9
% \end{array}\]

\subsection{Approach}
Very roughly the main tool behind the proofs of our results is to write the chromatic polynomial $Z(G;q)$ as the sum of two other polynomials $Z_1(G;q)+Z_2(G;q)$ which can be iteratively computed for all series-parallel graphs, see Section~\ref{sec:SP} for the precise definitions. We also define the rational function $R(G;q):=\frac{Z_1(G;q)}{Z_2(G;q)}$ and clearly $R(G;q)=-1$ implies $Z(G;0)=0$. A certain converse also holds under some additional conditions. 

To prove Theorem~\ref{thm:(0,32/27)} we essentially show that these rational functions avoid the value $-1$.
To prove presence of zeros we use that if the family of rational functions $\{q\mapsto R(G;q)\}$ behaves chaotically (formally, not being a \emph{normal family} near some parameter $q_0$, see Section~\ref{sec:active}), then one can use the celebrated Montel theorem from complex analysis to conclude that there must be a nearby value $q$ and a graph $G$ for which $Z(G,q)=0.$

Our approach to obtaining density of chromatic zeros is similar in spirit to Sokal's approach~\cite{Sokaldense}, but deviates from it in the use of Montel's theorem. 
Sokal uses Montel's `small' theorem to prove the Beraha-Kahane-Weis theorem~\cite{BKW78}, which he is able to apply to the generalized theta graphs because their chromatic polynomials can be very explicitly described.
It is not clear to what extent this applies to more complicated graphs.
Our use of Montel's theorem is however directly inspired by~\cite{de2021zeros}, which in turn builds on~\cite{PetersRegts,BGGS20,Buys}.
Our approach in fact also allows us to give a relatively short alternative proof for density of chromatic zeros of generalized theta graphs outside the disk $B_1(1)$, see Corollary~\ref{cor:|q-1|>1 theta}. %On the other hand the robustness comes at a cost in the sense that we are not able to get a precise statement like~\cite[Theorem 1.4]{Sokaldense}.

Our proof of Theorem~\ref{thm:leaf joined} makes use of an observation of Sokal and Royle in the appendix of the arXiv version of~\cite{RoyleSokal} (see \texttt{https://arxiv.org/abs/1307.1721}), saying that a particular recursion for ratios of leaf joined trees is up to a conjugation exactly the recursion for ratios of independence polynomial on trees.
We make use of this observation to build on the framework of~\cite{de2021zeros} allowing us to utilize some very recent work~\cite{bencs2021limit} giving an accurate description of the location of the zeros of the independence polynomial for the family of graphs with a given maximum degree.
%\feri{For the small finite $\Delta$ values (see Table~\ref{tab:over_Delta} we use the result of \cite{Buys} to find a family of leaf joint trees (as well trees) for which the closure of chromatic zeros (resp. independence zeros) contains a number with real part bigger than $\Delta$ (resp. strictly inside the cardioid).}\jeroen{This can now go to Section~\ref{sec:questions} right?}

\subsection*{Organization}
The next section deals with formal definitions of series-parallel graphs and ratios. We also collect several basic properties there that are used in later sections.
Section~\ref{sec:absence} is devoted to proving Theorem~\ref{thm:(0,32/27)}. In Section~\ref{sec:active} we state a general theorem allowing us to derive various results on presence of chromatic zeros for series-parallel graphs. Finally in Section~\ref{sec:leaf joined} we prove Theorem~\ref{thm:leaf joined}.
We end the paper with some questions in Section~\ref{sec:questions}

\section{Recursion for ratios of series-parallel graphs}\label{sec:SP}
We start with some standard definitions needed to introduce, and set up some terminology for series-parallel graphs.
We follow Royle and Sokal~\cite{RoyleSokal} in their use of notation.
%`We refer to~\cite{} for further background.

Let $G_1$ and $G_2$ be two graphs with designated start- and endpoints $s_1,t_1$, and $s_2,t_2$ respectively, referred to as \emph{two-terminal graphs}.
The \emph{parallel composition} of $G_1$ and $G_2$ is the graph $G_1\parallel G_2$ with designated start- and endpoints $s,t$ obtained from the disjoint union of $G_1$ and $G_2$ by identifying $s_1$ and $s_2$ into a single vertex $s$ and by identifying $t_1$ and $t_2$ into a single vertex $t.$
The \emph{series composition} of $G_1$ and $G_2$ is the graph $G_1\bowtie G_2$ with designated start- and endpoints $s,t$ obtained from the disjoint union of $G_1$ and $G_2$ by identifying $t_1$ and $s_2$ into a single vertex and by renaming $s_1$ to $s$ and $t_2$ to $t$. Note that the order matters here.
A two-terminal graph $G$ is called \emph{series-parallel} if it can be obtained from a single edge using series and parallel compositions. 
From now on we will implicitly assume the presence of the start- and endpoints when referring to a two-terminal graph $G$.
We denote by $\mathcal{G}_{\mathrm{SP}}$ the collection of all series-parallel graphs and by $\mathcal{G}^*_{\mathrm{SP}}$ the collection of all series-parallel graphs $G$ such that the vertices $s$ and $t$ are not connected by an edge.

Recall that for a positive integer $q$ and a graph $G=(V,E)$ we have
\[
Z(G;q)=\sum_{\phi:V\to \{1,\ldots,q\}} \prod_{uv\in E}(1-\delta_{\phi(u),\phi(v)}),
\]
where $\delta_{i,j}$ denotes the Kronecker delta.
For a positive integer $q$ and a two-terminal graph $G$, we can thus write\footnote{This can be seen to be the deletion-contraction relation for $G\parallel K_2$ with $Z^{\mathrm{dif}}(G;q)=Z(G\parallel K_2;q)$.},
\begin{equation}\label{eq:same dif}
Z(G;q)=Z^{\mathrm{same}}(G;q)^{}+Z^{\mathrm{dif}}(G;q),
\end{equation}
where $Z^{\mathrm{same}}(G;q)$ collects those contribution where $s,t$ receive the same color and where $Z^{\mathrm{dif}}(G;q)$ collects those contribution where $s,t$ receive the distinct colors. 
Since $Z^{\mathrm{dif}}(G;q)$ is equal to $Z(G\parallel K_2;q)$, where $K_2$ denotes an edge, 
%where $G'$ is obtained from $G$ by adding an edge between $s$ and $t$, 
both these terms are polynomials in $q$. Therefore \eqref{eq:same dif} also holds for any $q\in \mathbb{C}$.

We next collect some basic properties of $Z$, $Z^{\mathrm{same}}$ and $Z^{\mathrm{dif}}$ under series and parallel compositions in the lemma below. 
They can for example also be found in~\cite{Sokaldense}.
\begin{lemma}\label{lem:basic}
Let $G_1$ and $G_2$ be two two-terminal graphs and let us denote by $K_2$ an edge. Then we have the following identities:
\begin{itemize}
 \item[(P1)]   $Z^{\mathrm{dif}}(G;q) = Z(G \parallel K_2 ;q)$,
    \item[(P2)] $Z^{\mathrm{same}}(G_1 \bowtie G_2 ;q) = Z(G_1 \parallel G_2 ;q)$,
  \item[(P3)]  $Z(G_1 \bowtie G_2 ;q) = \tfrac{1}{q} \cdot Z(G_1 ;q) \cdot Z(G_2 ;q)$,
   \item[(P4)] $Z^{\mathrm{same}}(G_1 \parallel G_2 ;q) = \tfrac{1}{q}\cdot Z^{\mathrm{same}}(G_1 ;q) \cdot Z^{\mathrm{same}}(G_2 ;q)$,
  \item[(P5)]  $Z^{\mathrm{dif}}(G_1 \parallel G_2 ;q) = \tfrac{1}{q(q-1)}\cdot Z^{\mathrm{dif}}(G_1 ;q) \cdot Z^{\mathrm{dif}}(G_2 ;q)$,
  \item[(P6)]  $Z^{\mathrm{same}}(G_1 \bowtie G_2 ;q)= \tfrac{1}{q}\cdot Z^{\mathrm{same}}(G_1 ;q) \cdot Z^{\mathrm{same}}(G_2 ;q)+\tfrac{1}{q(q-1)}\cdot Z^{\mathrm{dif}}(G_1 ;q) \cdot Z^{\mathrm{dif}}(G_2 ;q)$,
  \item[(P7)] $Z^{\mathrm{dif}}(G_1 \bowtie G_2 ;q) =\tfrac{1}{q}\cdot Z^{\mathrm{same}}(G_1;q)\cdot Z^{\mathrm{dif}}(G_2;q)+\tfrac{1}{q}\cdot Z^{\mathrm{dif}}(G_1;q)\cdot Z^{\mathrm{same}}(G_2;q)$\\
    $+\tfrac{q-2}{q(q-1)}\cdot Z^{\mathrm{dif}}(G_1;q)\cdot Z^{\mathrm{dif}}(G_2;q).$
\end{itemize}
%\jeroen{P6 and P7 should be fixed now.}
\end{lemma}

%We note that $Z^{\mathrm{same}}(G;q)$ is equal to $Z(G',q,0)$, where $G'$ is the (multi)graph obtained from $G$ by identifying $s$ and $t$. 
%Thus for $G\in \mathcal{G}_{\mathrm{SP}}$, $Z^{\mathrm{same}}(G;q)$ is constantly equal zero if and only if $G\notin\mathcal{G}*_{\mathrm{SP}}$.

An important tool in our analysis of absence/presence of complex zeros is the use of the \emph{ratio} defined as
\begin{equation}\label{eq:ratio}
R(G;q):=\frac{Z^{\mathrm{same}}(G;q)}{Z^{\mathrm{dif}}(G;q)},
\end{equation}
which we view as a rational function in $q$.
We note that in case $G$ contains an edge between $s$ and $t$, the rational function $q\mapsto R(G;q)$ is constantly equal to $0$.
We observe that if $R(G;q)=-1$, then $Z(G;q)=0$ and the converse holds provided $Z^{\mathrm{dif}}(G;q)\neq 0$.

The next lemma provides a certain strengthening of this observation for series-parallel graphs.
\begin{lemma}\label{lem:Z=0 R-=-1}
Let $q\in \mathbb{C}\setminus \{0,1,2\}.$ Then the following are equivalent
\begin{itemize}
    \item[(i)] $Z(G;q)=0$ for some $G\in \mathcal{G}_{\mathrm{SP}}$,
    \item[(ii)] $R(G;q) = -1$ for some $G\in \mathcal{G}^*_{\mathrm{SP}}$,
    \item[(iii)] $R(G;q)\in\{0,-1,\infty\}$ for some $G\in \mathcal{G}^*_{\mathrm{SP}}$.
\end{itemize}
\end{lemma}

\begin{proof}
Throughout the proof we will refer to the properties stated in Lemma~\ref{lem:basic} without explicitly mentioning the lemma each time.

We start with `(i) $\Rightarrow$ (ii)'.
Let $q$ be as in the statement of the lemma such that $Z(G;q)=0$ for some series-parallel graph $G\in \mathcal{G}_{\mathrm{SP}}$. 
Take such a graph $G$ with as few edges as possible.

By the discussion between equation~\eqref{eq:ratio} and the statement of the present lemma, we may assume that $Z^{\mathrm{dif}}(G;q)=0$, for otherwise $R(G;q)=-1$ (and hence $G\in \mathcal{G}^*_{\mathrm{SP}}$).
Then also $Z^{\mathrm{same}}(G;q)= 0$.

Suppose first that $s,t$ are not connected by an edge.
By minimality, (P3) and (P4), $G$ must be the parallel composition of two series-parallel graphs $G_1$ and $G_2$ such that, say $Z^\mathrm{same}(G_1,q)=0$ and $G_1$ is not $2$-connected, or in other words such that $G_1$ is a series composition of two smaller series-parallel graphs $G_1'$ and $G_1''$. 
By (P2) we have that $Z(G_1'\parallel G_1'';q)=0$.
% If we now identify vertices $s$ and $t$ of $G_1$ we obtain a series-parallel graph $G'$ as the parallel composition of $G'_1$ and $G''_1$ (where vertices $s'_1$ and $t'_1$ have their roles reversed) for which $Z(G',q)=0$. 
This is a contradiction since $G_1'\parallel G_1''$ has fewer edges than $G$. We conclude that $R(G;q)=-1$ in this case.

Suppose next that $s$ and $t$ are connected by an edge. 
We shall show that we can find another series-parallel graph $\hat{G}\in \mathcal{G}^*_{\mathrm{SP}}$, that is isomorphic to $G$ as a graph (and hence has $q$ as zero of its chromatic polynomial) but not as two-terminal graph.
By the argument above we then have $R(\hat{G};q)=-1$.

Let $G'$ be obtained from $G$ by removing the edge $\{s,t\}$. 
Then by (P1) $Z^{\mathrm{dif}}(G';q)=Z(G;q)=0$.
If $Z^{\mathrm{same}}(G';q)=0$, then $Z(G';q)=0$, contradicting the minimality of $G$. Therefore $Z^{\mathrm{same}}(G';q)\neq 0$.
If $G'$ is the parallel composition of $G_1$ and $G_2$, then by (P5), \[Z^{\mathrm{dif}}(G_1;q)Z^{\mathrm{dif}}(G_2;q)=q(q-1)Z^{\mathrm{dif}}(G';q)=0,\] so there is a smaller graph, (namely $G_1\parallel K_2$ or $G_2\parallel K_2$), where $q$ is a zero, contradicting our choice of $G$.
Hence $G'$ is the series composition of two graphs $G_1$ and $G_2$. 
The graphs $G_1$ and $G_2$ cannot both be single edges, for otherwise $G$ would be a triangle and we excluded the values $q=0,1,2$.
So let us assume that $G_1$ is not a single edge. 
We will now construct $G$ in a different way as series-parallel graph. 
First switch the roles of $s_2$ and $t_2$ in $G_2$ and denote the resulting series-parallel graph by $G_2^T$. Then put $G^T_2$ in series with a single edge, and then put this in parallel with $G_1$. In formulas this reads as $\hat{G}:=(K_2 \bowtie G_2^T)\parallel G_1$. 
The resulting graph $\hat{G}$ is then isomorphic to $G$ (but not equal to $G$ as a two-terminal graph). % Let us denote the resulting graph by $\hat{G}$. 
In case $\hat{G}$ is not contained in $\mathcal{G}_{SP}^*$, then $G_1$ is also not in $\mathcal{G}_{SP}^*$. 
In that case let $G'_2$ be obtained from $G_2^T$ by first taking a series composition with an edge and then a parallel composition with an edge, that is, $G_2'=(K_2 \bowtie G_2^T)\parallel K_2$. 
We then have by (P1) and (P5),
\begin{align*}
Z(G;q)=Z(\hat{G};q)=Z^{\mathrm{dif}}(\hat{G};q)&=\tfrac{1}{q(q-1)}Z^{\mathrm{dif}}(G_1;q)Z^{\mathrm{dif}}(K_2 \bowtie G_2^T;q)
\\
&=\tfrac{1}{q(q-1)}Z(G_1;q)Z(G_2';q),
\end{align*}
% This follows from (P5) using that for any two-terminal graph $H$ we have $Z^\mathrm{dif}(H;q)=Z(H\parallel K_2;q)$.
So $q$ must be a zero of $Z(G_1;q)$, or of $Z(G_2';q)$. Because $G_1$ is not an edge, both $G_1$ and $G_2'$ contain fewer edges than $G$ contradicting the choice of $G$.
Hence we conclude that $\hat{G}$ is contained in  $\mathcal{G}_{SP}^*$, finishing the proof of the first implication.

The implication `(ii) $\Rightarrow$ (iii)' is obvious.
So it remains to show `(iii) $\Rightarrow$ (i)'. 

To this end suppose that $R(G;q)\in \{-1,0,\infty\}$ for some series-parallel graph $G\in \mathcal{G}^*_{\mathrm{SP}}$.
If the ratio equals $-1$, then clearly $Z(G;q)=0$. 
So let us assume that the ratio equals $0$.
Then $Z^{\mathrm{same}}(G;q)=0$ and we may assume that $Z^{\mathrm{dif}}(G;q)\neq 0$. 
Let us take such a graph $G$ with the smallest number of edges.
By minimality, $G$ cannot arise as the parallel composition of two series-parallel graphs $G_1$ and $G_2$ by (P4) and (P5). 
Therefore $G$ must be equal to the series composition of two series-parallel graphs $G_1$ and $G_2$.
Now, as in the proof of `(i) $\Rightarrow$ (ii)', identify vertices $s$ and $t$ of $G$ to form a new series-parallel graph $G'$, such that $Z(G';q)=Z^{\mathrm{same}}(G;q)=0$.

Let us finally consider the case that the ratio is equal to $\infty$.
In this case $Z^{\mathrm{dif}}(G;q)=0$.
Then by (P1), $Z(G\parallel K_2;q)=Z^{\mathrm{dif}}(G;q)=0$ and we are done.
\end{proof}

We next provide a description of the behavior of the ratios under the series and parallel compositions.
% The ratio of a single edge is the constant zero function $0$. 
% Now given two two-terminal graphs $G_1,G_2$ with ratios $R_1$ and $R_2$ respectively, putting them in parallel gives a 
% ratio 
% \begin{equation}\label{eq:ratio parallel}
% R=(q-1)R_1R_2, 
% \end{equation}
% and putting them in series gives a ratio
% \begin{equation}\label{eq:ratio series}
% R=\frac{(q-1)R_1R_2 + 1}{(q-1)(R_1+R_2) + q-2}.
% \end{equation}
To simplify the calculations, we will look at the modified ratio 
\begin{equation}\label{eq:def effective}
y_G(q):=(q-1)R(G;q),
\end{equation}
which, loosely following Sokal~\cite{Sokaldense}, we call the \emph{effective edge interaction}. 
\begin{remark}\label{rem:no weird things can happen}
Observe that $y_G(q)$ cannot be equal to any of the functions $q\mapsto -1,q\mapsto \infty$ and $q\mapsto 1-q$, since the numerator, $(q-1)Z^{\mathrm{same}}(G;q)$, and the denominator, $Z^{\mathrm{dif}}(G;q)$, have the same degree and leading coefficient, unless $G$ has an edge connecting $s$ and $t$, in which case $y_G(q)$ is the constant $0$ function.
\end{remark}

Given $q_0\in \mathbb{C}$ define
\begin{equation}
 \mathcal{E}(q_0):=\{y_G(q_0)\mid G \in \mathcal{G}_{SP}\},   
\end{equation}
%\jeroen{Does the $\mathcal{G}$ need to be starred here as well? I guess not, because 0 is our starting point.}
the set of all values of the effective edge interaction at $q_0$ for the family of series-parallel graphs as a subset of the Riemann sphere, $\hat{\mathbb{C}}=\mathbb{C}\cup \{\infty\}$.
As an example note that $0\in \mathcal{E}(q_0)$ for any $q_0$, being the effective edge interaction of a single edge.

For any $q\neq 0$ define the following M\"obius transformation\footnote{Readers familiar with the Tutte polynomial will recognize this formula as expressing the $x$-coordinate from the $y$-coordinate (or the other way around) on the hyperbola $(x-1)(y-1)=q$, on which, for positive integer $q$, the Tutte polynomial corresponds to the $q$-state Potts model partition function. See e.g.~\cite{Sokalsurvey} for more on the connection between the Tutte polynomial and the Potts model.} 
\[
y\mapsto f_q(y):=1+\frac{q}{y-1}\]
and note that $f_q$ is an involution, i.e. $f_q(f_q(y))=y$ for all $y$.

The next lemma captures the behavior of the effective edge interactions under series and parallel compositions and can be easily derived from Lemma~\ref{lem:basic}.
\begin{lemma}\label{lem:formulas effective}
% Let $G_1,G_2$ be two two-terminal graphs with effective edge interactions $y_1=y_{G_1}(q),y_2=y_{G_2}(q)$ respectively. 
% Denote $y_{\mathrm{ser}}$ and $y_{\mathrm{par}}$ for the effective edge interactions of the series and parallel composition of $G_1$ and $G_2$ respectively.
% Then
% \begin{align*}
% y_{\mathrm{par}}&=y_1y_2 & \text{if } \{y_1,y_2\}&\neq \{0,\infty\},
% \\
% y_{\mathrm{ser}}&=f_q(f_q(y_1)f_q(y_2)) & \text{if } \{y_1,y_2\} &\neq \{1,1-q\}.
% \end{align*}
Let $G_1,G_2$ be two two-terminal graphs. Then 
\begin{align*}
y_{G_1\parallel G_2}&=y_{G_1}y_{G_2},
\\
y_{G_1\bowtie G_2}&=f_q(f_q(y_{G_1})f_q(y_{G_2})).
 \end{align*}
Moreover, for any fixed $q_0\in \mathbb{C}$, if $\{y_{G_1}(q_0),y_{G_2}(q_0)\}\neq \{0,\infty\}$, then
\[
y_{G_1\parallel G_2}(q_0)=y_{G_1}(q_0)y_{G_2}(q_0),
\]
and if $ \{y_{G_1}(q_0),y_{G_2}(q_0)\} \neq \{1,1-q_0\}$, then
\[
y_{G_1\bowtie G_2}(q_0)=f_{q_0}(f_{q_0}(y_{G_1}(q_0))f_{q_0}(y_{G_2}(q_0))).
\]
\end{lemma}
We include a proof of the lemma for convenience of the reader.
\begin{proof}
First of all we note that the product $y_{G_1}y_{G_2}$ is always a well-defined rational function.
By Remark~\ref{rem:no weird things can happen}, $f_q(y_{G_i})$ cannot be constant $0$, but could be constant $\infty$. Therefore the product $f_q(y_{G_1})f_q(y_{G_2})$ could be constant $\infty$, but applying $f_q$ once more to it results again in a well-defined rational function.

The statements for the parallel connections follow directly from (P4) and (P5) from Lemma~\ref{lem:basic} and the definition of the effective edge interaction.
For the statements for the series connections let us denote $y_1=y_{G_1},y_2=y_{G_2}$ and $y_{\mathrm{ser}}=y_{G_1\bowtie G_2}$. 
We use (P6) and (P7) from Lemma~\ref{lem:basic} to write $y_{\mathrm{ser}}=\frac{y_1y_2+q-1}{y_1+y_2+q-2}$.
It is then not difficult to see that $f_q(y_{\mathrm{ser}})=f_q(y_1)f_q(y_2).$
%\begin{align*} 
%f_q(y_{\mathrm{ser}})
% &= \frac{\frac{y_1y_2+q-1}{y_1+y_2+q-2}+q-1}{\frac{y_1y_2+q-1}{y_1+y_2+q-2}-1}\\
% &= \frac{y_1y_2+q-1+(q-1)(y_1+y_2+q-2)}{y_1y_2+q-1-(y_1+y_2+q-2)}\\
% &= \frac{(y_1+q-1)(y_2+q-1)}{(y_1-1)(y_2-1)}
% &= f_q(y_1)f_q(y_2)
%\end{align*}
Therefore, since $f_q$ is an involution, 
\begin{equation*}%\label{eq:yser}
y_{\mathrm{ser}}=f_q(f_q(y_{\mathrm{ser}}))=f_q(f_q(y_1)f_q(y_2)),
\end{equation*}
as desired.
The statements for the evaluation at a fixed value $q_0\in \mathbb{C}$ now follow directly.
\end{proof}

\begin{remark}
Note that this lemma allows us to compute the effective edge interaction of any series-parallel graph. 
For example, the effective edge interaction of the path on three vertices, $P_2$, can be computed as 
\[
y_{P_2}=y_{K_2\bowtie K_2}=f_q(f_q(0)^2)=f_q((1-q)^2)=\frac{q-1}{q-2}.
\]
\end{remark}

% \jeroen{This doesn't tell you how to compute interactions, if one of the interactions is $\infty$. And actually, there is a possibility that $y_1=1-q, y_2=1$, in which you cannot compute $y_{ser}$ (because you're doing $0\cdot \infty$). As we don't mention $\mathcal{E}(q_0)$ in this Lemma, it might be a little bit better to talk about the interactions as rational functions, instead of complex numbers, because we can always do the multiplications then. There is then one small issue left, namely that $f_q(1)=\infty$ is not a rational function. This will only be an issue if $f_q(y)$ is the constant 0 function for some $y$ (in all other cases the product with $\infty$ is again $\infty$), which happens when $y=1-q$ (equality as rational functions). But this is impossible, as it would give a graph for which the chromatic polynomial is constant 0.}

% \jeroen{In section \ref{sec:absence} we really talk about the values of the interactions at a fixed $q$, so we want a version of the lemma with values, where we basically say that we can compute the interaction as long as we don't take $0\cdot\infty$. In section \ref{sec:active} we consider them as rational functions, so it's useful to have a rational function version where we don't have to bother checking if we do $0\cdot \infty$.}

\section{Absence of zeros near (0,32/27)} \label{sec:absence}
%\begin{theorem}\label{thm:(0,32/27)}
%There exists an open set $U$ containing the interval $(0,32/27)$ such that for any $q\in U\setminus\{1\}$ and any series-parallel graphs $G$, $Z(G;q)\neq 0$.
%\end{theorem}
In this section we prove Theorem~\ref{thm:(0,32/27)}.
In the proof we will use the following condition that guarantees absence of zeros and check this condition in three different regimes. We first need a few quick definitions.

For a set $S\subseteq \C$, denote $S^2:=\{s_1s_2\mid s_1,s_2\in S\}.$ 
For subsets $S,T$ of the complex plane, we use the notation $S\Subset T$ (and say $S$ is strictly contained in $T$) to say that the closure of $S$ is contained in the interior of $T$. 
For $r>0$ we define $B_r$ to be the closed disk of radius $r$ centered at $0$. 
\begin{lemma}\label{lem:condition for absence of zeros SP}
Let $q\in\C\setminus\{0,1,2\}$ and let $V\subseteq \C$ be a set satisfying: $0\in V$, $1-q\notin V^2$, $V^2\subseteq V$ and $f_q(f_q(V)^2)\subseteq V$. Then $Z(G;q)\neq0$ for all series-parallel graphs $G$.
\end{lemma}

\begin{proof}
By Lemma~\ref{lem:Z=0 R-=-1} it suffices to show that the ratios avoid the point $-1$.
Or equivalently, since $q\neq 1$, that the effective edge interactions at $q$ avoid the point $1-q$.

We will do so by proving the following stronger statement:
\begin{equation}\label{eq:inclusion}
 \mathcal{E}(q)\subseteq V \text{ and } 1-q\notin \mathcal{E}(q).
\end{equation}
We show this by induction on the number of edges. 
The base case follows since $0\in V$ and $q\neq 1$.
Assume next that $y\in\mathcal{E}(q)\setminus\{0\}$ and suppose that $y$ is the effective edge interaction of some series-parallel graph $G$.
If $G$ is the parallel composition of two series-parallel graphs $G_1$ and $G_2$ with effective edge interactions $y_1$ and $y_2$ respectively, then, by induction, $y_1,y_2\in V$ and neither of them is equal to $1-q$. 
By Lemma~\ref{lem:formulas effective} and our assumption we have $y=y_1y_2\in V^2\subseteq V$. Since $1-q\notin V^2$, we also have that $y\neq 1-q$.
If $G$ is the series composition of two series-parallel graphs $G_1$ and $G_2$ with effective edge interactions $y_1$ and $y_2$ respectively, then, by induction, $y_1,y_2\in V$ and neither of them is equal to $1-q$. 
Therefore $f_q(y_i)\neq 0$ for $i=1,2$.
Then by Lemma~\ref{lem:formulas effective} and our assumption, $y=f_q(f_q(y_1)f_q(y_2))\in V$.
Moreover, $f_q(1-q)=0\neq f_q(y_1)f_q(y_2)=f_q(y)$. Therefore $y\neq 1-q$.
This shows~\eqref{eq:inclusion} and finishes the proof.

\end{proof}

Below we prove three lemmas allowing us to apply the previous lemma to different parts of the interval $(0,32/27)$. First we collect two useful tools.
For two complex numbers $a,b$ we denote by $C(a,b)$ the circle in the complex plane with the line segment between $a$ and $b$ as a diameter. In case $a=b$, $C(a,b)$ consists of the single point $\{a\}$.

\begin{lemma}\label{lem:real circles}
Let $q,r\in\R$, then the circle $C(r,f_q(r))$ is $f_q$-invariant.
\end{lemma}
\begin{proof}
First note that $f_q$ maps the real line to itself, because $q$ is real.
Now let $C=C(r,f_q(r))$. %be the circle with diameter the line segment between $r$ and $f_q(r)$. 
Then $C$ intersects the real line at right angles.
The M\"obius transformation $f_q$ sends $C$ to a circle through $f_q(r),f_q(f_q(r))=r$, and because $f_q$ is conformal the image must again intersect the real line at right angles. Therefore $f_q(C)=C$.
\end{proof}

\begin{proposition}\label{prop: V^2=V^2}
Let $V\subseteq \C$ be a disk. 
Then
\[
V^2 = \{y^2 \mid y\in V\}.
\]
\end{proposition}
\begin{proof}
Obviously the second is contained in the first. The other inclusion is an immediate consequence of the Grace-Walsh-Szeg\H{o} theorem.
\end{proof}

Now we can get into the three lemmas mentioned.

\begin{lemma}\label{lem:(0,1)}
For each $q\in(0,1)$ there exists a closed disk $V\subseteq\C$ strictly contained in $B_{\sqrt{1-q}}$, satisfying $0\in V$, $f_q(V)=V$ and $V^2 \Subset V$.
\end{lemma}

\begin{proof}
Let $r=\sqrt{1-q}$ and choose real numbers $a\in(r^2,r), b\in(-r,-r^2)$ with $f_q(a)=b$. They exist because $f_q(r)=-r$ and $f'_q(r)=\frac{-q}{(1-r)^2}<0$.
Let $V$ be the closed disk with diameter the line segment between $a$ and $b$. Clearly $V\Subset B_r$ and $0\in V$.
From Lemma~\ref{lem:real circles} it follows that the boundary of $V$ is mapped to itself. Further, the interior point $0\in V$ is mapped to $f_q(0)=1-q=r^2$ which is also an interior point of $V$. Therefore $f_q(V)=V$.
Last, we see that $V^2 \subseteq B_r^2=B_{r^2} \Subset V$, confirming all properties of $V$.
\end{proof}

\begin{lemma}\label{lem:(32/27)}
For each $q\in (1,32/27)$ there exists a closed disk $V\subseteq \mathbb{C}$ strictly contained in $B_{\sqrt{q-1}}$ satisfying $0\in V$, $f_q(V)=V$ and  $V^2\Subset V$.
\end{lemma}
\begin{proof}
The equation $f_q(z)=z^2$ has a solution in $(-1/3,0)$, since $f_q(0)=1-q<0$ and $f_q(-1/3)=1-3q/4>1/9$. Denote one such solution as $r$. Then we see that 
\begin{equation}\label{eq:der smaller than 2r}
f'_q(r)=\frac{-q}{(r-1)^2}=-r-1<2r=[z^2]'_{z=r},
\end{equation}
%Since $r\in (-1/3,0)$ we have $r^3>-1/3r^2$ and $-r>3r^2.$
%By definition of $f_q$ and $r$ we have $1+q/(r-1)=r^2$, from which it follows that
and
\begin{equation}\label{eq:q-1 smaller than r^2}
q-1=r^3-r^2-r>-\tfrac{1}{3}r^2-r^2+3r^2>r^2.
\end{equation}
Since $f_q(r)=r^2<-r$, it follows that for $t\in (-1/3,r)$ close enough to $r$ we have $f_q(t)<-t$, $t^2<f_q(t)$ by \eqref{eq:der smaller than 2r} and $t>-\sqrt{q-1}$ by \eqref{eq:q-1 smaller than r^2}.
%Let $r$ be the largest negative solution of 
%\[
%    z^2=1+\frac{q}{z-1}=f_q(z),
%\]
%or equivalently 
%\[
%    (z-1)^2(z+1)=q.
%\]
%Such a solution exists, since $(z-1)^2(z+1)$ is a concave function on $[-1,0]$ and at the points $-1,-1/3,0$ the polynomial $(z-1)^2(z+1)$ has values $0,32/27,0$ respectively. 
%Additionally we obtain that $r>-1/3$.
%
%Now we claim that 
%\[
%r^2<q-1.
%\] 
%This is true, since
% $q-1=r^3-r^2-r$ and $0<r^3-2r^2-r=r(r^2-2r-1)$ for any $r\in[-1/3,0)$.
%
%Together this implies that there exists a  $t\in (-1,0)$, such that 
%\begin{itemize}
%    \item $t^2<f_q(t)<-t$,
%    \item $t>-\sqrt{q-1}$.
%\end{itemize}
Fix such a value of $t$ and let $V$ be the closed disk with diameter the line segment between $t$ and $f_q(t)$. The exterior point $\infty$ is now mapped to the exterior point $1$, so by Lemma~\ref{lem:real circles} we then know that $f_q(V)=V$.
By construction we have that 
\[
V^2\subseteq B^2_t=B_{t^2} \Subset B_{f_q(t)}\subseteq V
%[B_{|t|} \supseteq V\supseteq  B_{f_q(t)}(0) \supsetneq B_{t^2}(0)=B^2_{|t|}(0)\supseteq V^2. 
\]
and so $V$ satisfies the desired properties.
\end{proof}

\begin{lemma}\label{lem:around 1}
There exists an open neighborhood $I$ around $1$ such that for each $q\in I\setminus\{1\}$ there exists a disk $V\subseteq \mathbb{C}$,
satisfying $0\in V$, $1-q\not\in V^2$, $V^2\subseteq V$ and $f_q(f_q(V)^2) \subseteq V$.
%strictly contained in $B_{\sqrt{|q-1|}}$ with center at zero satisfying $V^2\subsetneq V$ and  $f_q(V)^2\subseteq f_q(V)$.
\end{lemma}

\begin{proof}
Let $R=\sqrt{|1-q|}$. We claim that if $R$ is sufficiently small, there exists an $0<s<R$ such that $V=B_s$ satisfies the required conditions. Actually, we will show this to be true with $R<2-\sqrt{3}$, thus giving for $I$ the open disk $|q-1|<7-4\sqrt{3}$.

Trivially, $0\in V, 1-q\not\in V^2$ and $V^2\subseteq V$, so we only need to show that $f_q(f_q(V)^2)\subseteq V$, or equivalently $f_q(V)^2 \subseteq f_q(V)$.

We start with bounding the image of the disk $B_s$:
\begin{align*}
    f_q(B_s)&=\left\{\frac{y+q-1}{y-1} ~\middle|~ y\in B_s\right\}\\
    &\subseteq \left\{\frac{y+q'-1}{y'-1} ~\middle|~ y,y'\in B_s, q'\in B_{R^2}(1)\right\}\\
        &\subseteq \left\{\frac{z}{y'-1} ~\middle|~ y'\in B_s, z\in B_{R^2+s}\right\}\\
    &\subseteq \left\{z ~\middle|~ |z|\le \frac{R^2+s}{1-s}\right\}.
\end{align*}
So if we define $\rho(s)=\frac{R^2+s}{1-s}$, then $f_q(B_s)\subseteq B_{\rho(s)}$. 
Since $f_q$ is an involution, we have
\[
    B_{\rho^{-1}(s)}\subseteq f_q(B_s).    
\]
Now we claim that if $R<2-\sqrt{3}$, then there exists $0<s<R$ such that $\rho(s)^2<\rho^{-1}(s)$.
This is sufficient since for this value of $s$ we have
\[
f_q(B_s)^2\subseteq B^2_{\rho(s)}=B_{\rho(s)^2}\subseteq B_{\rho^{-1}(s)}\subseteq f_q(B_s),
\]
as desired.

We now prove the claim.
As $0<s<R<1$, the inequality $\rho(s)^2<\rho^{-1}(s)=\frac{s-R^2}{1+s}$ is equivalent to
\begin{align*}
    (R^2+1)(3s^2+(R^2-1)s+R^2)&<0, & 0&<s<R.
\end{align*}
%\begin{align*}
%    (R^2 + 1) (R^2 (s + 1) + s (3 s - 1)) &<0,&  0&<s<R,
%    \end{align*}
%    which is equivalent to
%    \begin{align*}
%    3s^2+s(R^2-1)+R^2&<0, &  0&<s<R.
%\end{align*}
If we have a solution, then the quadratic polynomial in the variable $s$ should have $2$ real solutions, since its main coefficient is positive. Since the linear term is negative and the constant term is positive, both roots are positive. 
Thus it is sufficient to prove that  the ``smaller'' real root is less then $R$, i.e.
\[
    \frac{(1-R^2)-\sqrt{(1-R^2)^2-12R^2}}{6}<R.
\]
This indeed holds true for $R<2-\sqrt{3}$.
%This is the case if %\feri{Cases depending on $1-6R-R^2$ is negative or positive, but all in all}
%\[
%    R\in(0,2-\sqrt{3}).
%\]
%This means that we can take our set $I$ to be a disk of radius $(2-\sqrt{3})^2$ centered at $1$.
\end{proof}

Now we are ready to prove Theorem~\ref{thm:(0,32/27)}.
\begin{proof}[Proof of Theorem~\ref{thm:(0,32/27)}]
For every $q\in(0,32/27)$ we will now find an open $U$ around $q$, such that $U\setminus\{1\}$ does not contain chromatic zeros of series-parallel graphs. For $q=1$ this follows directly from Lemmas~\ref{lem:around 1} and \ref{lem:condition for absence of zeros SP}. For $q\in(0,1)$ and $q\in(1,32/27)$ we appeal to Lemmas~\ref{lem:(0,1)} and \ref{lem:(32/27)} respectively to obtain a closed disk $V$ with $V\Subset B_{\sqrt{|1-q|}}$, $f_q(V)=V$ and $V^2\Subset V$. We then claim that there is an open $U$ around $q$, for which this disk $V$ still satisfies the requirements of Lemma~\ref{lem:condition for absence of zeros SP} for all $q'\in U$.\\
Certainly $0\in V$ and $V^2 \subseteq V$ remain true. Because $V\Subset B_{\sqrt{|1-q|}}$ holds, we can take $U$ small enough such that $V\subseteq B_{\sqrt{|1-q'|}}$ still holds, which confirms $1-q' \not\in V^2$. Lastly, we know that $f_q(f_q(V)^2) =f_q(V^2) \Subset f_q(V)=V$. Because $V$ is compact, and the function $y\mapsto f_{q'}(f_{q'}(y)^2)$ depends continuously on $q'$, the inclusion $f_{q'}(f_{q'}(V)^2)\Subset V$ remains true on a small enough open $U$ around $q$.
\end{proof}

\section{Activity and zeros}\label{sec:active}
In this section we prove Theorems~\ref{thm:q>32/27} and \ref{thm:Re(q)>3/2}. 
We start with a theorem that gives a concrete condition to check for presence of chromatic zeros.
For any $q\neq 0$ we call any $y\in f_q(\mathcal{E}(q))$ a \emph{virtual interaction}.
For example, $f_{q}(0)=1-q$ is a virtual interaction (obtained from the effective edge interaction of a single edge).

\begin{theorem}\label{thm:to escape or not to escape}
Let $q_0\in \mathbb{C}\setminus \{0\}.$
If there exists either an effective edge interaction $y\in \mathcal{E}(q_0)$ or a virtual interaction $y\in f_{q_0}(\mathcal{E}(q_0))$ such that $|y|>1$, then there exists $q$ arbitrarily close to $q_0$ and $G\in \mathcal{G}_{\mathrm{SP}}$ such that $Z(G;q)=0$.
\end{theorem}
We will provide a proof for this result in the next subsection. First we consider some corollaries.

The first corollary recovers a version of Sokal's result~\cite{Sokaldense}.% and directly implies Proposition~\ref{prop:|q-1|>1}.
\begin{corollary}\label{cor:|q-1|>1}
Let $q\in \mathbb{C}$ such that $|1-q|>1$. 
Then there exists $q'$ arbitrarily close to $q$ and $G\in  \mathcal{G}_{\mathrm{SP}}$ such that $Z(G;q')=0$.
\end{corollary}
\begin{proof}
%\jeroen{This can now be simplified, because $1-q$ is immediately a virtual interaction which satisfies the conditions.}
First of all note that as mentioned above, $y=f_q(0)=1-q$, is a virtual interaction (since $0$ is the effective edge interaction of a single edge).
By assumption we thus have a virtual interaction $y$ such that $|y|>1.$
%Denote by $P_n$ the path of length $n$. 
%Then the effective edge interaction of $P_n$ is given by $f_q(f_q(0)^n)$.
%Now note that $f_q(\partial B(1,1)$ is a line splitting the complex plane into two open half planes $H_1$ and $H_2$. 
%Since $1-q\notin \mathbb{R}$ and since $f_q(\infty)=1$ it follows that the orbit $\{(1-q)^n\mid n\in \mathbb{N}\}$ intersects both $H_1$ and $H_2.$
%This implies that there is some $n$ such that $|f_q(f_q(0)^n)|>1$.
The result now directly follows from Theorem~\ref{thm:to escape or not to escape}.
\end{proof}
\begin{remark}
Recall that a generalized theta graph is the  parallel composition of a number of equal length paths.
Sokal~\cite{Sokaldense} in fact showed that we can take $G$ in the corollary above to be a generalized theta graph.
Our proof of Theorem~\ref{thm:to escape or not to escape} in fact also gives this. We will elaborate on this in Corollary~\ref{cor:|q-1|>1 theta} after giving the proof.
\end{remark}

Our second corollary gives us Theorem~\ref{thm:q>32/27}.
\begin{corollary}\label{cor:q>32/27}
Let $q>32/27$. Then there exists $q'$ arbitrarily close to $q$ and $G\in \mathcal{G}_{\mathrm{SP}}$ such that $Z(G;q')=0$.
\end{corollary}
\begin{proof}
Consider the map $g(z)=f_q(z^2)$. We claim that $g(z)<z$ for any $z\in(-1,0]$. As $g(0)=1-q<0$, it is sufficient to show that $g(z)\neq z$ for any $z\in(-1,0)$. Or equivalently,
\begin{align*}
    q\neq (z-1)^2(z+1).
\end{align*}
The maximal value of $(z-1)^2(z+1)$ on the interval $(-1,0]$ is $32/27$ (which is achieved at $-1/3$), thus the claim holds.

We next claim that there exists $k$ such that $g^{\circ k}(0)\leq -1$. 
(Here $g^{\circ k}$ denotes the $k$-fold iterate of the map $g$.)
Suppose not, then since the sequence $\{g^{\circ k}(0)\}_{k\geq 0}$ is decreasing it must have a limit $L$.
By construction, $L\in[-1,0]$ and it must be a fixed point of the map $g$. Since $\lim_{z\to -1^{+}}g(z)=-\infty$, it follows that $g$ has no fixed points in $[-1,0]$, a contradiction.

We also claim that $g^{\circ k}(0)$ is an element of $\mathcal{E}(q) \cup f_q(\mathcal{E}(q))$ for any integer $k\geq0$.
Indeed this follows by induction, the base case being $k=0$.
Assuming that $g^{\circ i}(0)\in \mathcal{E}(q)$ for some $i\geq 0$, it follows that $g^{\circ i}(0)^2\in  \mathcal{E}(q)$ by Lemma~\ref{lem:formulas effective} and therefore $g^{\circ i+1}(0)\in f_q(\mathcal{E}(q))$.
And similarly, if $g^{\circ i}(0)\in f_q(\mathcal{E}(q))$ for some $i\geq 0$, it follows that $g^{\circ i}(0)=f_q(y)$ for some $y\in \mathcal{E}(q)$ and hence by Lemma~\ref{lem:formulas effective}, $g^{\circ i+1}(0)=f_q(f_q(y)^2)\in  \mathcal{E}(q)$.

To finish the proof, we choose $k\in\mathbb{N}$ such that $g^{\circ k}(0)\leq -1$.
If the inequality is actually strict, so $g^{\circ k}(0) < -1$, the result now directly follows from Theorem~\ref{thm:to escape or not to escape}, since $g^{\circ k}(0)$ is an element of $\mathcal{E}(q)\cup f_q(\mathcal{E}(q))$.
If on the other hand $g^{\circ k}(0)=-1$, then $g^{\circ k+1}(0)=\infty$. For even $k$, we see that $g^{\circ k}(0)$ is an effective interaction. As a rational function of $q$, it cannot be constant $-1$ by Remark~\ref{rem:no weird things can happen}. So the value of $g^{\circ k}(0)$ for some $q'$ arbitrarily close to $q$ is outside the unit disk and we again apply Theorem~\ref{thm:to escape or not to escape}. For odd $k$ we see that $g^{\circ k+1}(0)$ is an effective interaction and cannot be constant $\infty$, again by Remark~\ref{rem:no weird things can happen}.
Hence there again exists $q'$ arbitrarily close to $q$ where the value is finite and outside the unit disk and we again can apply Theorem~\ref{thm:to escape or not to escape}.
\end{proof}

Our next corollary gives us Theorem~\ref{thm:Re(q)>3/2}.

\begin{corollary}\label{cor:Re(q)>3/2}
Let $q\in \mathbb{C}$ such that $\Re(q) > 3/2$. Then there exists $q'$ arbitrarily close to $q$ and $G\in \mathcal{G}_{\mathrm{SP}}$ such that $Z(G;q')=0$.
\end{corollary}
\begin{proof}
Consider the path $P_2$ of length $2$, which is the series composition of two single edges.
Therefore, by Lemma~\ref{lem:formulas effective} its effective edge interaction is given by 
\[
f_q(f_q(0)^2)=f_q((1-q)^2)=\frac{q-1}{q-2}.
\]
Now the M\"obius transformation $q\mapsto \frac{q-1}{q-2}$ maps the half plane $\{z\mid \Re(z)\geq 3/2\}$ to the complement of the unit disk, since $\infty\mapsto 1$, $3/2\mapsto -1$ and the angle that the image of $\{z\mid \Re(z)= 3/2\}$ makes with $\mathbb{R}$ at $-1$ is $90$ degrees and since $0\mapsto 1/2$.
The result now directly follows from Theorem~\ref{thm:to escape or not to escape}.
\end{proof}

%\feri{Doesn't imply that zeros of graphs from $\Theta$  are also dense here?}

\subsection{Proof of Theorem~\ref{thm:to escape or not to escape}}

 We first introduce some definitions inspired by~\cite{de2021zeros}.
Let $\mathcal{G}$ be a family of two-terminal graphs. Let $q_0\in \hat{\mathbb{C}}$. 
Then we call $q_0$ \emph{passive for $\mathcal{G}$} if there exists an open neighborhood $U$ around  $q_0$ such that the family of ratios $\{q\mapsto R(G;q)\mid G\in \mathcal{G}\}$ is a normal family on $U$, that is, if any infinite sequence of ratios contains a subsequence that converges uniformly on compact subsets of $U$ to a holomorphic function $f:U\to \hat{\mathbb C}$.
We call $q_0$ \emph{active for $\mathcal{G}$} is $q_0$ is not passive for $\mathcal{G}$.
We define the \emph{activity locus of $\mathcal{G}$} by
\begin{equation}
    \mathcal{A}_\mathcal{G}:=\{q_0\in \hat{\mathbb{C}}\mid q_0 \text{ is active for } \mathcal{G}\}.
\end{equation}
Note that the activity locus is a closed subset of $\hat{\mathbb{C}}$.

We next state Montel's theorem, see~\cite{CarlesonGamelinBook,MilnorBook} for proofs and further background.
\begin{theorem}[Montel] \label{thm:montel}
Let $\mathcal{F}$ be a family of rational functions on an open set $U\subseteq \hat{\mathbb{C}}$. 
If there exists three distinct points $a,b,c\in \hat{\mathbb{C}}$ such that for all $f\in \mathcal{F}$ and all $u\in U$, $f(u)\notin \{a,b,c\}$, then $\mathcal{F}$ is a normal family on $U$.
\end{theorem}

Montel's theorem combined with activity and Lemma~\ref{lem:Z=0 R-=-1} give us a very quick way to demonstrate the presence of chromatic zeros.

\begin{lemma}\label{lem:activity gives zeros}
Let $q_0\in \mathbb{C}\setminus\{0,1,2\}$ %\jeroen{We apply Lemma~\ref{lem:Z=0 R-=-1}, so we should exclude $0,1,2$.} 
and suppose that $q_0$ is contained in the activity locus of $\mathcal{G}_{\mathrm{SP}}$.
Then there exists $q$ arbitrarily close to $q_0$ and $G\in \mathcal{G}_{\mathrm{SP}}$ such that $Z(G;q)=0$.
\end{lemma}
\begin{proof}
Suppose not. 
Then by Lemma~\ref{lem:Z=0 R-=-1}, there must be an open neighborhood of $q_0$ on which family of ratios must avoid the points $-1,0,\infty$.
Montel's theorem then gives that the family of ratios must be normal on this neighborhood, contradicting the assumptions of the lemma.
\end{proof}

%\begin{lemma}\jeroen{This lemma is now less important for the results, so it can be moved.}
%Let $q\in\C$ with $|1-q|<1$. Then $\bar{\mathcal{E}(q)}=\bar{f_q(\mathcal{E}(q))}$.
%\end{lemma}
%\begin{proof}
%The sequence $f_q(0)^n=(1-q)^n$ converges to 0, and is contained in $f_q(\mathcal{E}(q))$. Therefore 0 is contained in the closure. Now the iterative description of $\mathcal{E}(q)$ shows that $\mathcal{E}(q) \subseteq \bar{f_q(\mathcal{E}(q))}$. Applying the continuous map $f_q$ yields $f_q(\mathcal{E}(q)) \subseteq f_q\left(\bar{f_q(\mathcal{E}(q))}\right)\subseteq \bar{\mathcal{E}(q)}$. Together this implies the equality.
%\end{proof}

\begin{lemma}\label{lem:activity outside unit disk}
Let $q_0\in\mathbb{C}$, and assume there exists an effective edge interaction $y\in\mathcal{E}(q_0)$ or a virtual interaction $y\in f_{q_0}(\mathcal{E}(q_0))$ such that $|y|>1$. Then $q_0$ is contained in the activity locus of $\mathcal{G}_{\mathrm{SP}}$.
\end{lemma}
\begin{proof}
We will show that for every open $U'$ around $q_0$ there exists a family of series-parallel graphs $\mathcal{G}$ such that $\{q\mapsto y_G(q) \mid G\in\mathcal{G} \}$ is non-normal. 
This of course implies non-normality of the family $\{q\mapsto R(G;q) \mid G\in\mathcal{G}\}$ on $U'$ and hence that $q_0$ is contained in the activity locus $\mathcal{A}_{\mathcal{G}_{\mathrm{SP}}}$.

We will first assume that $y\in f_{q_0}(\mathcal{E}(q_0))$ and $|y|>1$.
 Suppose $y=f_{q_0}(y_{q_0}(G)))$ for some series-parallel graph $G$. The virtual interaction is not a constant function of $q$, because at $q=\infty$ the virtual interaction is $\infty$, cf. Remark~\ref{rem:no weird things can happen}. Therefore any open neighborhood $U'$ of $q_0$ is mapped to an open neighborhood $U$ of $y$ and we may assume that $U'$ is small enough, such that $U$ lies completely outside the closed unit disk. Now the pointwise powers $\{u^n\mid u\in U\}_{n\in\mathbb{N}}$ converge to $\infty$ and the complex argument of the powers $\arg(\{u^n\mid u\in U\})=n \arg(U)$ cover the entire unit circle for $n$ large enough.

Let us denote the unit circle by $C\subseteq\mathbb{C}$. 
Then $f_q(C)$ is a straight line through $1$ for every $q$. Inside the Riemann sphere, $\hat{\mathbb{C}}$, these lines are circles passing through $\infty$. For $U'$ small enough and $q\in U'$, and in a neighborhood of $\infty$, these circles will lie in two sectors. More precisely, there exists $R$ large enough such that the argument of the complex numbers in $\bigcup_{q\in U'}f_q(C) \cap \{z\in\mathbb{C} \mid |z| >R\}$ are contained in two small intervals. Therefore we can find two sectors $S_1$ and $S_2$ around $\infty$ such that $f_q(S_1)$ lies inside $C$ for all $q\in U'$ and $f_q(S_2)$ lies outside of $C$ for all $q\in U'$. Because the pointwise powers $\{u^n\mid u\in U\}$ converge towards $\infty$ and the argument of the complex numbers are spread over the entire unit circle, there must be an $N$ for which $\{u^N\mid u\in U\}$ intersects with both $S_1$ and $S_2$. Then $\{f_q(f_q(y_G(q))^N) \mid q\in U'\}$ has points inside and outside the unit circle. Now the family $\{q\mapsto f_q(f_q(y_G(q))^N)^m \mid m\in\mathbb{N}\}$ is non-normal on $U'$.
Indeed, the values inside the unit circle converge to $0$, and the values outside the unit circle converge to $\infty$. So any limit function of any subsequence can therefore not be holomorphic.
An easy induction argument, as in the proof of Corollary~\ref{cor:q>32/27}, shows that $f_q(f_q(y_G(q))^N)^m$ is the effective edge interaction of the parallel composition of $m$ copies of the series composition of $N$ copies of the graph $G$.

For the case $y\in \mathcal{E}(q_0)$ with $|y|>1$, we note again that this interaction cannot be a constant function of $q$, because at $q=\infty$ the value must be $1$, cf. Remark~\ref{rem:no weird things can happen}. 
If we perform the same argument as above, we obtain a non-normal family of virtual interactions on $U'$. 
Applying $f_q$ to this family, produces a non-normal family on $U'$ of effective edge interactions of series compositions of copies of parallel compositions of copies of the graph $G$.

%In both cases, we can conclude that $q_0$ is in the activity locus $\mathcal{A}_{\mathcal{G}_{SP}}$.
\end{proof}

\begin{remark}\label{rem:active family of graphs}
For later reference we record the family of graphs that provides the non-normal family of interactions/ratios. 
In the case that we have a virtual interaction $|f_{q_0}(y_G(q_0))|>1$ for a graph $G$, the family consists of $N$ copies of $G$ in series, and $m$ copies of this in parallel.
For the case of an effective edge interaction $|y_G(q_0)|>1$, we instead put $N$ copies of $G$ in parallel, and $m$ copies of this in series.
\end{remark}

\begin{proof}[Proof of Theorem~\ref{thm:to escape or not to escape}]
For $q\in\mathbb{C}\setminus\{0,1,2\}$ where either the interaction or the virtual interaction escapes the unit disk, the theorem is a direct consequence of Lemmas~\ref{lem:activity gives zeros} and \ref{lem:activity outside unit disk}. If for $q\in\{0,1,2\}$ there is an interaction or virtual interaction escaping the unit disk, this holds for all $q'$ in a neighborhood as well. At these values, we already know that zeros accumulate, so they will accumulate at $q$ as well.
\end{proof}

We now explain how to strengthen Corollary~\ref{cor:|q-1|>1} to generalized theta graphs. Let $\Theta$ denote the family of all generalized theta graphs.
\begin{corollary}\label{cor:|q-1|>1 theta}
Let $q\in \mathbb{C}$ such that $|1-q|>1$. 
Then there exists $q'$ arbitrarily close to $q$ and $G\in  \Theta$ such that $Z(G;q')=0$.
\end{corollary}
\begin{proof}
Note that $y=f_q(0)=1-q$ is a virtual activity such that $|y|>1.$
From Lemma~\ref{lem:activity outside unit disk} and Remark~\ref{rem:active family of graphs} we in fact find that $q$ is in the activity locus of $\Theta$.
By Theorem~\ref{thm:montel} (Montel's theorem) we may thus assume that there exists $G\in \Theta$ such that $R(G;q)\in \{-1,0,\infty\}$.
We claim that the ratio must in fact equal $-1$, meaning that $q$ is in fact a zero of the chromatic polynomial of the generalized theta graph $G$.

The argument follows the proof of `(iii) $\Rightarrow$ (i)' in Lemma~\ref{lem:Z=0 R-=-1}.
Suppose that the ratio is $\infty$. 
Then we add an edge between the two terminals and realize that the resulting graph is equal to a number cycles glued together on an edge. Since chromatic zeros of cycles are all contained in $B_1(1)$, this implies that the ratio could not have been equal to $\infty$. 
If the ratio equals $0$, then we again obtain a chromatic zero of a cycle after identifying the start and terminal vertices.
This proves the claim and hence finishes the proof.
\end{proof}

\section{Chromatic zeros of leaf joined trees from independence zeros}\label{sec:leaf joined}
This section is devoted to proving Theorem~\ref{thm:leaf joined}.
Fix a positive integer $\Delta\geq 2$ and write $d=\Delta-1$.
Given a rooted tree $(T,v)$ consider the two-terminal graph $\hat T$ obtained from $(T,v)$ by identifying all leaves (except $v$) into a single vertex $u.$ We take $v$ as the start vertex and $u$ as the terminal vertex of $\hat T$. 
Following Royle and Sokal~\cite{RoyleSokal}, we call $\hat T$ a \emph{leaf joined tree}.
%Given such a triple let us denote by $T_{/U}$ the graph obtained from $T$ by identifying all vertices in $U$ to a single vertex.
We abuse notation and say that a leaf joined tree $\hat T$ has maximum degree $\Delta=d+1$ if all its vertices except possibly its terminal vertex have degree at most $\Delta.$
We denote by $\mathcal{T}_d$ the collection of leaf joined trees of maximum degree at most $d+1$ for which the start vertex has degree at most $d$.

%It is not hard to see that $T_{U}$ is a series-parallel graph.
Our strategy will be to use Lemma~\ref{lem:Z=0 R-=-1} in combination with an application of Montel's theorem, much like in the previous section.
To do so we make use of an observation of Royle and Sokal in the appendix of the arXiv version of~\cite{RoyleSokal} saying that ratios of leaf joined trees, where the underlying tree is a Cayley tree, are essentially the \emph{occupation ratios} (in terms of the independence polynomial) of the Cayley tree. 
We extend this relation here to all leaf-joined trees and make use of a recent description of the zeros of the independence polynomial on bounded degree graphs of large degree due to the first author, Buys and Peters~\cite{bencs2021limit}.

%In the Potts model we will interpret the set $U$ to be fixed to colour 1. We define the ratio
%\[
%R_{T,v,U}(q,y) \coloneqq \frac{Z_{T,\mathrm{Potts}}^{U\to1, v\to 1}(q,y)}{Z_{T,\mathrm{Potts}}^{U\to1, v\not\to 1}(q,y)}.
%\]

%Recall that for a graph $G=(V,E)$ we define
%\[
%Z_{G,\mathrm{Potts}}(q,y) = \sum_{f:V \to [q]} y^{\#\text{monochromatic edges for }f} = \sum_{A\subseteq E} (y-1)^{|A|}q^{\#\text{components of }(V,A)}.
%\]
%We use superscripts to denote boundary conditions for the colour function $f$.

\subsection{Ratios and occupation ratios}
For a graph $G=(V,E)$ the independence polynomial in the variable $\lambda$ is defined as
\begin{equation}\label{eq:def ind}
    I(G;\lambda)=\sum_{\substack{I\subseteq V\\I \text{ ind.}}}\lambda^{|I|},
\end{equation}
where the sum ranges over all sets of $G$. 
(Recall that a set of vertices $I\subseteq V$ is called \emph{independent} if no two vertices in $I$ form an edge of $G$.)
We define the \emph{occupation ratio} of $G$ at $v\in V$ as the rational function
\begin{equation}\label{eq:ind ratio}
    P_{G,v}(\lambda):=\frac{\lambda I(G\setminus N[v];\lambda)}{I(G-v;\lambda)},
\end{equation}
where $G-v$ (resp. $G\setminus N[v]$) denotes the graph obtained from $G$ by removing $v$ (resp. $v$ and all its neighbors).
We define for a positive integer $\Delta$, $\mathcal{G}_\Delta$ to be the collection of rooted graphs $(G,v)$ of maximum degree at most $\Delta$ such that the root vertex, $v$, has degree at most $d:=\Delta-1$.
We next define the relevant collection of occupation ratios,
\[
\mathcal{P}_\Delta:=\{ P_{G,v}\mid (G,v)\in \mathcal{G}_\Delta\}.
\]
A parameter $\lambda_0\in \mathbb{C}$ is called \emph{active for $\mathcal{G}_\Delta$} if the family $\mathcal{P}_\Delta$ is not normal at $\lambda_0$.

We will use the following alternative description of $\mathcal{P}_\Delta$.
Define 
\[
F_{\lambda,d}(z_1,\ldots,z_{d})=\frac{\lambda}{\prod_{i=1}^d(1+z_i)}
\]
and let $\mathcal{R}_{\lambda,d}$ be the family of rational maps, parametrized by $\lambda$, and defined by 
\begin{itemize}
    \item[(i)] the identify map $z\mapsto z$ is contained in $\mathcal{R}_{\lambda,d}$
    \item[(ii)] if $r_1,\ldots,r_d\in \mathcal{R}_{d,\lambda}$, then $F_{\lambda,d}(r_1(z),\ldots,r_d(z))\in \mathcal{R}_{\lambda,d}$.
\end{itemize}
\begin{lemma}[Lemma 2.4 in~\cite{bencs2021limit}] \label{lem:ind ratios}
Let $\Delta\geq 2$ be an integer and write $d=\Delta-1$.
Then 
\[
\mathcal{P}_\Delta=\{\lambda\mapsto r_\lambda(0)\mid r_\lambda\in \mathcal{R}_{\lambda,d}\}. 
\]
\end{lemma}

We will next show that, up to a simple factor, the occupation ratios of graphs of maximum degree at most $\Delta$ are contained in the family of chromatic ratios of leaf joined tree of maximum degree at most $\Delta.$ Define 
\[\lambda(q,d):=\frac{(q-1)^d}{(q-2)^{d+1}}.
\]

\begin{proposition}\label{prop:chromatic ratios}
Let $\Delta\geq 2$ be a positive integer and write $d=\Delta-1$.
Then
\[
\left\{q\mapsto \frac{q-2}{q-1}r_{\lambda(q,d)}(0)\mid r_\lambda\in \mathcal{R}_{\lambda,d}\right\}\subseteq \{q\mapsto R(\hat{T};q)\mid (T,v)\in \mathcal{T}_d\}.
\]
\end{proposition}
\begin{proof}
Suppose that $r_\lambda\in \mathcal{R}_{\lambda,d}$ and that $r_\lambda(z)=F_{\lambda,d}(r_{\lambda;1}(z),\ldots,r_{\lambda;d}(z))$ for certain $r_{\lambda;i}\in \mathcal{R}_{\lambda,d}$.
We need to show that the map $q\mapsto \frac{q-2}{q-1}r_{\lambda(q,d)}(0)$ is equal to the ratio $R(\hat{T};q)$ for some rooted tree $(T,v)\in \mathcal{T}_d$.
By induction we may assume that there are leaf joined trees $\hat{T_1},\ldots \hat{T_d}\in \mathcal{T}_d$ such that there exists
$r_{\lambda;i}\in \mathcal{R}_{\lambda,d}$ for each $i=1,\ldots,d$ such that
\begin{equation}\label{eq:induction hyp}
q\mapsto \frac{q-2}{q-1}r_{\lambda(q,d);i}(0)=R(\hat{T_i};q).
\end{equation}
Note that the base case is covered since the map $q\mapsto 0$ is the ratio of the edge $\{v,u\}$.

%For $i=k+1,\ldots,d$ consider the triple $(\{v_i\},v_i,\emptyset).$ and 
Let $(T_1,v_1),\ldots,(T_d,v_d)$ be the underlying rooted trees of the $\hat{T_i}$.
Let $\hat{T}$ be the leaf joined tree whose underlying rooted tree $(T,v)$ is obtained from $(T_1,v_1),\ldots,(T_d,v_d)$ by adding a new root vertex $v$ and connecting it to all the $v_i$.
We claim that
\begin{equation}\label{eq:claim ratio}
    R(\hat{T};q)=q\mapsto \frac{q-2}{q-1}r_{\lambda(q,d),d}(0).
\end{equation}
To prove this we will first compute the effective edge interaction of $\hat{T}.$
To do so observe that $\hat{T}$ is obtained by first putting $K_2$ in series with $\hat{T_{i}}$ for $i=1,\ldots,d$ and then putting the resulting graphs in parallel. (Incidentally this shows that all leaf joined trees are series-parallel graphs).
In formulas this reads as
\begin{equation}\label{eq:construction T/U}
\hat{T}= (K_2\bowtie \hat{T_1})\parallel  (K_2\bowtie \hat{T_{2}}) \parallel \cdots \parallel  (K_2\bowtie \hat{T_{d}}).
\end{equation}
Suppose the graphs $\hat{T_i}$ have effective edge interaction $y_i$ ($i=1,\ldots,d$), then by Lemma~\ref{lem:formulas effective} $\hat{T}$ has effective interaction $y$ given by
\begin{equation}
y= \prod_{i=1}^d f_q(f_q(0)f_q(y_i))=\left(\frac{q-1}{q-2}\right)^d\prod_{i=1}^d\frac{1}{1+y_i/(q-2)}.
\end{equation}
Recall that $R(\hat{T};q)=y_G(q)/(q-1)$.
If we now define the \emph{modified ratio} $\widetilde{R}(G;q)=\frac{q-1}{q-2}R(G;q)$ for any two-terminal graph $G$, we can write this relation as
\begin{align*}
\widetilde{R}(\hat{T};q) &= \frac{\lambda(q,d)}{\prod_{i=1}^d (1+\widetilde{R}(\hat{T_i};q))}
\\
&=F_{\lambda(q,d),d}\left(\widetilde{R}(\hat{T_1};q),\ldots,\widetilde{R}(\hat{T_d};q)\right)
\\
&=r_{\lambda(d,q)}(0)
\end{align*}
by \eqref{eq:induction hyp}.
This finishes the proof.
\end{proof}

\begin{corollary}\label{cor:active implies active}
Let $\Delta\geq2$ be an integer and write $d=\Delta-1$.
Let $q_0\in \mathbb{C}\setminus \{1,2,1-d\}$.
If $\lambda(q_0,d)$ is active for $\mathcal{G}_\Delta$, then $q_0$ is active for ${\mathcal{T}_d}$.
\end{corollary}
\begin{proof}
Note that the derivative of $\lambda(q,d)$ with respect to $q$ is given by
\[
-(q+d-1)\frac{(q-1)^{d-1}}{(q-2)^{d+2}}.
\]
Therefore the map $q\mapsto \lambda(q,d)$ is injective on a neighborhood of $q_0$ and the result follows from the previous proposition.
\end{proof}

\subsection{Proof of Theorem~\ref{thm:leaf joined}}
We are now ready to harvest some results from~\cite{bencs2021limit} and provide a proof of Theorem~\ref{thm:leaf joined}.

Let for an integer $\Delta\geq 2$ and $u\in \mathbb{C}$
\[
\lambda_\Delta(u):=\frac{-(\Delta-1)^{\Delta-1}u}{(\Delta-1+u)^\Delta}
\]
and define
\[
\mathcal{C}_\Delta:=\left\{\lambda_\Delta(u) \mid |u|<1\right\}.
\]
Define the following collection of active parameters
\[
\mathcal{N}_\Delta:=\{u\in B_{1/2}(1/2)\mid  \text{ the family } \mathcal{P}_\Delta \text{ is not normal at } \lambda_\Delta(-u)\}.
\]
\begin{theorem}\label{thm:active ind}
There exists $\Delta_0>0$ such that for all integers $\Delta\geq \Delta_0$ the set $\mathcal{N}_\Delta$ contains a nonempty open set and in particular is nonempty.
\end{theorem}
\begin{proof}
This follows directly from \cite[Theorem 1.2 and 1.3]{bencs2021limit} combined with~\cite[Theorem 1]{de2021zeros} and the fact that the boundary of the set $\mathcal{U}_\infty$ (as defined in~\cite{bencs2021limit}) is not differentiable at $e$. 
Indeed, a close inspection of the function describing the part of the boundary with positive imaginary part near $e$ shows that it in fact makes an angle of $120$ degrees with the real axis.
%\feri{It is $120^\circ$.}
\end{proof}
We now give a proof of Theorem~\ref{thm:leaf joined}.
\begin{proof}[Proof of Theorem~\ref{thm:leaf joined}]
Let $\Delta_0$ from the theorem above.
Fix any integer $\Delta\geq \Delta_0$ and write $d=\Delta-1$.
Choose any non real $u_0\in \mathcal{N}_\Delta$.
Define $q_0=1+d/u_0$ and observe that since the M\"obius transformation $u\mapsto 1+d/u$ maps the disk $B_{1/2}(1/2)$ onto the half plane $\{z\in \mathbb{C}\mid \Re (z)\geq d+1\}$, it follows that $\Re (q_0)>\Delta$.
Furthermore,
\[
\lambda(q_0,d)=\frac{(q_0-1)^d}{(q_0-2)^{d+1}}=\frac{(d/u_0)^d}{((d-u_0)/u_0)^{d+1}}=\frac{d^du_0}{(d-u_0)^{d+1}}=\lambda_\Delta(-u_0).
\]
Therefore, by Corollary~\ref{cor:active implies active} and Theorem~\ref{thm:active ind}, we obtain that $q_0\in \mathcal{A}_{\mathcal{T}_d}$, the activity locus of the family of ratios of the leaf joined trees contained in $\mathcal{T}_d$.
By Theorem~\ref{thm:montel} (Montel's theorem) we conclude that there must exist $q$ such that $\Re (q)>\Delta$ and a leaf joined tree $\hat{T}\in \mathcal{T}_d$ such that $R(\hat{T};q)\in \{0,-1,\infty\}$.

We now show that there exists a leaf joined tree of maximum degree $\Delta$ for which $q$ is zero of its chromatic polynomial. 
We cannot directly invoke Lemma~\ref{lem:Z=0 R-=-1}, but its proof will essentially give us what we need.

If the ratio, $R(\hat{T};q)$, is equal to $-1$ then $Z(\hat{T};q)=0$.
If the ratio equals $\infty$ we add an edge between the two terminal vertices such that $q$ is a chromatic zero of the resulting leaf joined tree, whose maximum degree is still $\Delta$.
Finally, suppose the ratio equals $0$. 
We know that $\hat{T}$ is the parallel composition of $d$ leaf joined trees $\hat{T_i}$ each in series with $K_2$ (see \eqref{eq:construction T/U}). 
Since the ratio equals $0$ we know by Lemma~\ref{lem:basic} that  $Z^{\mathrm{same}}(K_2\bowtie\hat{T_i};q)=0$ for some $i$.
Now putting the graph $K_2\bowtie \hat{T_i}$ in parallel with an edge gives a new leaf joined tree $\hat{T_i'}$ of maximum degree $\Delta$ such that $Z(\hat{T_i'};q)=Z^{\mathrm{same}}(\hat{T_i};q)=0$.
This finishes the proof.
\end{proof}

\section{Concluding remarks, questions and conjectures}\label{sec:questions}
In this paper we embarked on the quest to determine the location of the chromatic zeros of the family of series-parallel graphs.
While we have made several contributions, a complete characterization remains elusive, as is visible in  
Figure~\ref{fig:zeroes+zerofree}. 
Several concrete questions and conjectures arise in this regard.

First of all, it is important to note that Figure~\ref{fig:zeroes+zerofree} is a pixel picture, and the color of a pixel only displays the behavior of the center point of the pixel. Potential features of the picture that are smaller than the resolution will therefore be invisible.
We believe however that with a bit more effort one can create a more rigorous picture that looks exactly the same.
A pixel is colored blue, if for $q$ at the center of the pixel, there exist integers $n_1,\ldots,n_k$ (within the search depth $n_1\cdot\ldots\cdot n_k \leq 300$) with $|f_q(f_q(\cdots f_q(f_q(0)^{n_1})^{n_2}\cdots)^{n_k})| >1$; the darkest shade of blue corresponds to a search depth of $75$ the lighter shades correspond to a depth of $150$ and $300$ respectively. 
This composition is either an effective edge interaction, or a virtual interaction, of a series-parallel graph with $n_1\cdot\ldots\cdot n_k$ edges and so Theorem~\ref{thm:to escape or not to escape} ensures that $q$ is contained in the closure of the chromatic zeros of series parallel graphs. 

Answering a question from a previous version of the present paper, the second author showed in his thesis that 
%Our first question concerns to what extend 
Theorem~\ref{thm:to escape or not to escape} actually gives a complete characterization of the chromatic zeros of series parallel graphs. 
More precisely, in~\cite[Theorem 2.26]{Huiben} he showed that if $Z(G;q)=0$ for some $q\in \mathbb{C}\setminus \{0,1,2\}$ for some $G\in \mathcal{G}_{\mathrm{SP}}$ then there exists $G'\in \mathcal{G}_{\mathrm{SP}}$ such that $1<|y_{G'}(q)|<\infty$.
%\begin{?}
%Let $q_0\in \mathbb{C}\setminus \{0,1,2\}$. 
%Is it true that if $Z(G;q_0)=0$ for some series-parallel graph, then there exists an effective or virtual edge %interaction $y$ such that $|y|>1$?
%\end{?}
%We note that a variant of this question is true for the independence polynomial~\cite{de2021zeros}.

A pixel is colored orange in Figure~\ref{fig:zeroes+zerofree}, if for $q$ at the center of the pixel, it is possible to find a disk $V$ such that $f_q(V)=V$ and which satisfies the conditions of Lemma~\ref{lem:condition for absence of zeros SP}. There is a very explicit description of the disks $V$ satisfying $f_q(V)=V$. This makes it easy to check $0\in V$ and $1-q\not\in V^2$. The condition $V^2\subseteq V$ is verified by checking that $\sup\{|z|^2 \mid z\in V\} < \inf\{|z| \mid z\in \mathbb{C}\setminus V \}$.
%\jeroen{Some more details: search depth, pixel size.}
Figure~\ref{fig:zeroes+zerofree} directly motivates the following conjecture.
\begin{conjecture}
For each $q$ in the punctured disk $B_{5/27}(1)\setminus\{1\}$ and any series-parallel graph $G$, $Z(G;q)\neq0$.
\end{conjecture}
Note that our proof of Lemma~\ref{lem:around 1} gives a punctured disk of radius $(2-\sqrt{3})^2\approx 0.072$ around $1$, which is much less than $5/27 \approx 0.185$.

Another interesting question motivated by Theorem~\ref{thm:leaf joined} is whether there exist chromatic zeros with real part larger than the second largest degree for all degrees.
We have verified this question up to $\Delta\le 45$, see Table~\ref{tab:over_Delta} below.
\begin{table}[h!]
    \centering
    \[\begin{array}{c|rcr|c}
            \Delta & &q& & \text{type} \\ 
            \hline
            4  &  4.027  &+&  0.783 i &  3 , 2     \\
5  &  5.088  &+&  0.836 i &  4 , 3     \\
6  &  6.132  &+&  0.881 i &  5 , 4     \\
7  &  7.058  &+&  1.521 i &  6 , 4     \\
8  &  8.120  &+&  1.577 i &  7 , 5      \\
9  &  9.012  &+&  2.194 i &  8 , 5     \\
10  &  10.084  &+&  2.256 i &  9 , 6   \\
11  &  11.147  &+&  2.314 i &  10 , 7  \\
12  &  12.038  &+&  2.928 i &  11 , 7  \\
13  &  13.109  &+&  2.990 i &  12 , 8   \\
14  &  14.173  &+&  3.049 i &  13 , 9  \\
15  &  15.063  &+&  3.662 i &  14 , 9  \\
16  &  16.133  &+&  3.724 i &  15 , 10 \\
17  &  17.197  &+&  3.784 i &  16 , 11 \\
18  &  18.087  &+&  4.395 i &  17 , 11 \\
19  &  19.157  &+&  4.457 i &  18 , 12 \\
20  &  20.222  &+&  4.518 i &  19 , 13 \\
21  &  21.111  &+&  5.129 i &  20 , 13 \\
22  &  22.180  &+&  5.191 i &  21 , 14  \\
23  &  23.246  &+&  5.252 i &  22 , 15 \\
24  &  24.135  &+&  5.862 i &  23 , 15
\end{array}\quad
\begin{array}{c|rcr|c}
            \Delta & &q& & \text{type} \\ 
            \hline
25  &  25.204  &+&  5.925 i &  24 , 16 \\
26  &  26.269  &+&  5.986 i &  25 , 17 \\
27  &  27.158  &+&  6.596 i &  26 , 17 \\
28  &  28.227  &+&  6.658 i &  27 , 18 \\
29  &  29.293  &+&  6.719 i &  28 , 19 \\
30  &  30.182  &+&  7.329 i &  29 , 19 \\
31  &  31.251  &+&  7.392 i &  30 , 20 \\
32  &  32.317  &+&  7.453 i &  31 , 21 \\
33  &  33.206  &+&  8.063 i &  32 , 21 \\
34  &  34.274  &+&  8.125 i &  33 , 22 \\
35  &  35.340  &+&  8.187 i &  34 , 23  \\
36  &  36.229  &+&  8.796 i &  35 , 23 \\
37  &  37.298  &+&  8.859 i &  36 , 24 \\
38  &  38.364  &+&  8.920 i &  37 , 25  \\
39  &  39.252  &+&  9.530 i &  38 , 25  \\
40  &  40.321  &+&  9.592 i &  39 , 26 \\
41  &  41.387  &+&  9.654 i &  40 , 27 \\
42  &  42.276  &+&  10.263 i &  41 , 27\\
43  &  43.344  &+&  10.326 i &  42 , 28\\
44  &  44.411  &+&  10.387 i &  43 , 29\\
45  &  45.299  &+&  10.997 i &  44 , 29

    \end{array}\]
    
    \caption{Table of parameters $q$ with real part bigger than $\Delta$, such that $q$ is active for the following family of leaf joined trees:
    construct trees where alternately every vertex has down degree exactly $d_1$ resp. $d_2$, add $d_1-d_2$ leaves to the down vertices of degree $d_2$, and add one vertex connected to all leaves. The proof of Theorem~\ref{thm:leaf joined} implies  that chromatic zeros of leaf joined trees of maximum degree $d_1+1$ accumulate at $q$.}
    \label{tab:over_Delta}
\end{table}
The values were obtained using the technique of Buys~\cite{Buys} to find zeros of the independence polynomial. 
First we find a family of spherically regular trees of degree $d_1\ge d_2$ that are active at $\lambda\in \mathbb{C}$ for this family,  using Appendix~B of \cite{Buys}. 
Therefore by Corollary~\ref{cor:active implies active} we obtain that $q_0$ is active for $\mathcal{T}_{d_1}$, where we choose $q_0$ to be the solution of $\lambda(q_0,d_1)=\lambda$ of the  largest real part.

\begin{figure}[h!]
\begin{center}
\includegraphics{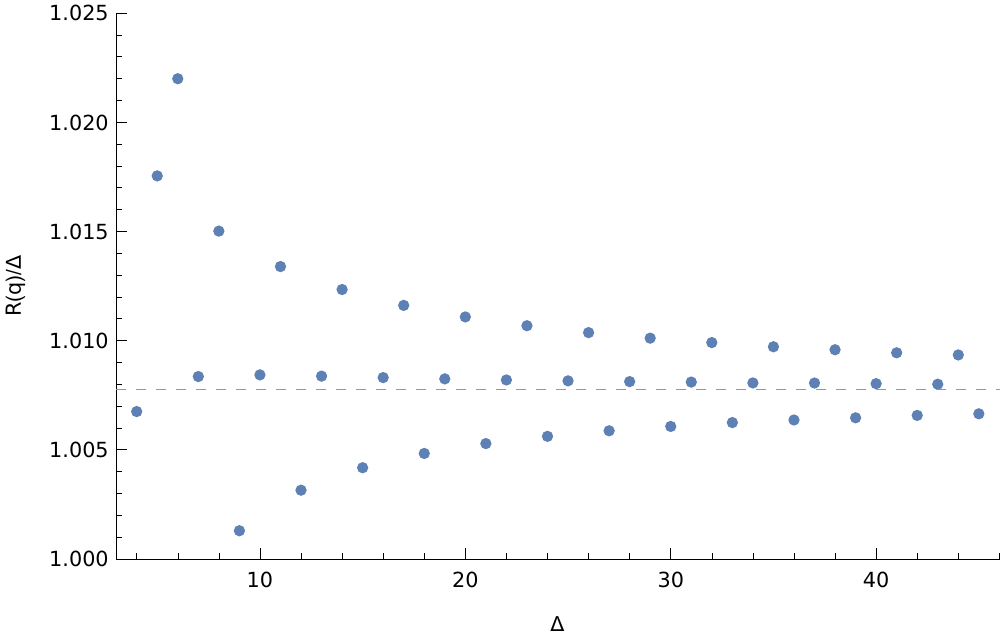}
\end{center}
\caption{For each $\Delta=4,\ldots,45$ we record the value of $\Re(q)/\Delta$ from Table~\ref{tab:over_Delta}. The orange dashed line denotes the limiting value as $d_1\to \infty$ and $d_2/d_1\to 2/3$.}
\label{f:data2/3}
\end{figure}
%\feri{Description for the Table is needed, how did we obtained it.} \jeroen{Do we also want the plot of $\Re(q)/\Delta$?}

Figure \ref{f:data2/3} strongly supports the following conjecture. This is related to a question from~\cite{PetersRegts,Buys} on zeros of the independence polynomial of bounded degree graphs.
\begin{conjecture}
Theorem~\ref{thm:leaf joined} is true with $\Delta_0=3$.
\end{conjecture}

We end with a question on the possible extension of one of our result to a larger family of graphs to which our techniques do not seem to apply.
\begin{?}
What can be said about planar or triangulated planar graphs? Is it true that there are no chromatic zeros for these graphs in a punctured open set containing the interval $(0,32/27)$?
\end{?}

%1.00776470375339

\section*{Acknowledgment}
We thank Pjotr Buys for useful discussions regarding the counterexamples found in Theorem~\ref{thm:leaf joined} and Table~\ref{tab:over_Delta}.
We moreover thank the anonymous referees for constructive feedback and we thank Fengmin Dong for pointing out an error in the proof of Corollary 2 in an earlier version of the paper.

%\newpage

%\bibliographystyle{plain}
%\bibliography{chromatic}

\end{document}